\documentclass[12pt]{amsart}
\usepackage{amsmath}
\usepackage{amssymb}
\newtheorem{theorem}{Theorem}
\newtheorem{proposition}[theorem]{Proposition}

\newtheorem{corollary}[theorem]{Corollary}
\newtheorem*{remark*}{Remark}

\newcommand{\C}{\Bbb C}
\newcommand{\D}{\Bbb D}
\newcommand{\G}{\Bbb G}
\newcommand{\N}{\Bbb N}
\newcommand{\Om}{\Omega}

\def\phi{\varphi}
\def\O{\mathcal O}
\def\M{\mathcal M}
\def\diag{\operatorname{diag}}
\def\ord{\operatorname{ord}}
\def\sp{\operatorname{sp}}
\def\tr{\operatorname{tr}}

\title[Spectral Nevanlinna-Pick and Carath\'eodory-Fej\'er problems]
{Spectral Nevanlinna-Pick and Carath\'eodory-Fej\'er problems for
$n\le 3$}

\begin{document}

\author{Nikolai Nikolov}
\address{Institute of Mathematics and Informatics\\ Bulgarian Academy of
Sciences\\1113 Sofia, Bulgaria} \email{nik@math.bas.bg }

\author{Peter Pflug}
\address{Carl von Ossietzky Universit\"at Oldenburg, Institut f\"ur Mathematik,
Postfach 2503, D-26111 Oldenburg, Germany}
\email{peter.pflug@uni-oldenburg.de}

\author{Pascal J. Thomas}
\address{Universit\'e de Toulouse\\ UPS, INSA, UT1, UTM \\
Institut de Math\'e\-ma\-tiques de Toulouse\\
F-31062 Toulouse, France} \email{pthomas@math.univ-toulouse.fr}

\thanks{This note was written during the stay of the first named author at the
University of
Oldenburg supported by a DAAD grant (November 2009 - January
2010).}



\begin{abstract} The Nevanlinna-Pick problem and the simplest case of the
Cara\-th\'eodory-Fej\'er problem on the spectral ball $\Om_3$ are
reduced to interpolation problems on the symmetrized three-disc
$\G_3.$
\end{abstract}

\maketitle

Let $\M_n$ be the set of all $n\times n$ complex matrices. For
$A\in\M_n$ denote by $\sp(A)$ and $r(A)=\max_{\lambda\in
\sp(A)}|\lambda|$ the spectrum and the spectral radius of $A$,
respectively. The \textit{spectral ball} $\Om_n$ is given as
$$
\Om_n:=\{A\in\M_n:r(A)<1\}.
$$

\section{The spectral Nevanlinna-Pick problem}

In this note we first study the \textit{spectral Nevanlinna-Pick
problem} (for short we will write (SNPP)) on $\Om_n$ for $n=2,3:$
\smallskip

\noindent\textit{Given $k$ points $\alpha_1,\dots,\alpha_k$ in the
open unit disc $\D\subset\C$ and $k$ matrices
$A_1,\dots,A_k\in\Om_n$, decide whether there exists a holomorphic
mapping $\phi\in\O(\D,\Om_n)$ with $\phi(\alpha_j)=A_j$, $1\leq
j\leq k$}.
\smallskip

\noindent This problem has been studied by many authors; we refer
to \cite{Agl-You1,Agl-You2,Ber,BFT,Cos} and the references there. We
should mention that whenever there is a solution for
$\alpha_j,A_j$, then there is one for $\alpha_j,\tilde A_j$, when
$A_j$ is similar to $\tilde A_j$ (see \cite{BFT}); a fact that will be strongly
used in the proofs. Recall that (SNPP) for $k$ matrices on
$\Omega_n$ is completely understood if $k=n=2$ or if all matrices
$A_j$ have singleton spectra (cf.~\cite{Ber}).

In this note we give a complete reduction of this problem to an
interpolation problem on the so-called \textit{symmetrized
polydisc} $\G_n$ for $n=2,$ respectively $n=3.$ Recall that $\G_n$
is defined by
$$\G_n:=\{\sigma(A):A\in\Om_n\},$$
where the mapping $\sigma=(\sigma_1,\dots,\sigma_n)$ is given
by the following formula:
$$
\det(tE-A)=\sum_{j=0}^n(-1)^j\sigma_j(A)t^{n-j}
$$
($\sigma_0(A)=1).$ This kind of reduction leads to a problem which
may be simpler because $\G_n$ is a bounded hyperconvex (in
particular, taut) domain of dimension $n$ (much less than the
dimension $n^2$ of $\Om_n$).

If we have a solution $\psi\in\O(\D,\Om_n)$ of (SNPP) with the
above data $\alpha_j,A_j$, i.e. $\psi(\alpha_j)=A_j$ for all $j$,
then $\phi:=\sigma\circ \psi\in\O(\D,\G_n)$ with
$\phi(\alpha_j)=\sigma(A_j)$ for all $j$. Note that then the
$\phi_j$ may satisfy additional relations.

We are interested to find exactly those conditions (necessary and
sufficient) for a $\phi\in\O(\D,\G_n)$  with
$\phi(\alpha_j)=\sigma(A_j)$ for $j=1,\dots,k$, such that there
exists a $\psi\in\O(\D,\Om_n)$ satisfying $\phi=\sigma\circ\psi$
and $\psi(\alpha_j)=A_j$ for $j=1,\dots,k$; this will imply a
reduction of (SNPP) on $\Om_n\subset\C^{n^2}$ to an interpolation
problem on $\G_n\subset\C^n$. In this note we are mainly
interested in the cases $n=2$ and $n=3$.

There are no conditions if all $A_j's$ are cyclic (even in
$\Om_n,$ cf.~\cite{Cos}). Recall that a matrix $A\in\M_n$ is said
to be \textit{cyclic} (or \textit{non-derogatory}) if it admits a
cyclic vector (for other equivalent properties see \cite{NTZ}).

The case $k=2$ and $A_1$ cyclic has been
studied in \cite{Tho-Tra}.

Now we formulate the complete reduction in the case $n=2$ (see
also \cite{Agl-You1,Agl-You2,Ber}). Note that $A\in\M_2$ is either
cyclic or scalar (i.e.~$A=\lambda I,$ where $\lambda\in\C$ and $I$
is the unit matrix.)

\begin{proposition}\label{1} Let $A_1=\lambda_1I,\dots,A_k=\lambda_kI\in\Om_2,$
let $A_{k+1},\dots,A_l\in\Om_2$ be non-scalar matrices, and let
$\phi\in\O(\D,\G_2)$  be such that $\phi(\alpha_j)=\sigma(A_j)$
for $j=1,\dots,l.$ Then there exists a $\psi\in\O(\D,\Om_2)$
satisfying $\phi=\sigma\circ\psi$ and $\psi(\alpha_j)=A_j$ for
$j=1,\dots,l$ if and only if
$\phi'_2(\alpha_j)=\lambda_j\phi'_1(\alpha_j)$ for $j=1,\dots,k.$
\end{proposition}

\begin{corollary}\label{2} Let $A_1,\dots,A_l$ be as above and let
$\alpha_1,\dots,\alpha_l\in\D.$
Then there is a $\psi\in\O(\D,\Om_2)$ such that
$\psi(\alpha_j)=A_j$ for $j=1,\dots,l$ if and only if there is a
$\phi\in\O(\D,\G_2)$ such that $\phi(\alpha_j)=\sigma(A_j)$ for
$j=1,\dots,l,$ and $\phi'_2(\alpha_j)=\lambda_j\phi'_1(\alpha_j)$
for $j=1,\dots,k.$
\end{corollary}

\begin{proof}[Proof of Proposition \ref{1}] The necessary part is clear
(setting $\phi=\sigma\circ\psi$). For the converse, note that $A_j$
is similar to its companion matrix
$\tilde
A_j=\left(\begin{array}{cc}0&1\\-\phi_2(\alpha_j)&\phi_1(\alpha_j)\end{array}\right)$
for $j=k+1,\dots,l.$ Set $\tilde A_j=A_j$ for $j=1,\dots,k.$
Let $P$ and $Q$ be polynomials with simple zeros such that
$$
P(\alpha_j)=\left\{\begin{array}{ll}
\lambda_j&j\le k\\
0&j>k\end{array}\right.,\quad Q(\alpha_j)=\left\{\begin{array}{ll}
0&j\le k\\
1&j>k\end{array}\right.$$ and $Q$ has no more zeros. Then the
conditions $\phi'_2(\alpha_j)=\lambda_j\phi'_1(\alpha_j)$ for
$j=1,\dots,k$ imply that the mapping
$$\tilde\psi=\left(\begin{array}{ccc}P&Q\\R&\varphi_1-P\end{array}\right),$$
where $R=\frac{P\varphi_1-P^2-\varphi_2}{Q},$ does the job for $\tilde A_j$ instead
of $A_j.$
It remains to set $\psi=e^{-f}\tilde\psi e^f,$ where $f:\C\to\M_2$ is a polynomial mapping
such that
$A_j=e^{-f(\alpha_j)}\tilde A_je^{f(\alpha_j)}.$
\end{proof}

\begin{remark*} {\rm Note that $\psi$ is bounded (and
$\psi(\zeta)$ is  cyclic for
$\zeta\in\D\setminus\{\alpha_1,\dots,\alpha_k\}$). Therefore,
Proposition \ref{1} also says that if there is a solution of
(SNPP), then there exists also a bounded one.}
\end{remark*}

Note that if $n\ge 3,$ then there exist non-cyclic but
non-scalar matrices in $\Om_n$. For $n=3$ the complete reduction
of (SNPP) is given in the following result.

\begin{proposition}\label{3} Let $A_1=\lambda_1 I,\dots,A_k=\lambda_k I\in\Om_3,$ let
$A_{k+1},\dots,A_l\in\Om_3$ be non-cyclic and non-scalar matrices
such that $\sp(A_j)=\{\lambda_j,\lambda_j,\mu_j\}$, $k+1\leq j\leq
l$. Moreover, let $A_{l+1},\dots,A_m\in\Om_3$ be cyclic matrices.
Let $\phi\in\O(\D,\G_3)$  be such that
$\phi(\alpha_j)=\sigma(A_j)$ for $j=1,\dots,m.$ Then there exists
a $\psi\in\O(\D,\Om_3)$ satisfying $\phi=\sigma\circ\psi$ and
$\psi(\alpha_j)=A_j$ for $j=1,\dots,m$ if and only if the
following conditions hold:

\noindent$\bullet$
$\phi'_2(\alpha_j)=2\lambda_j\phi'_1(\alpha_j),$
$\phi'_3(\alpha_j)=\lambda_j^2\phi_1'(\alpha_j),$ and
$\phi''_3(\alpha_j)-\lambda_j\phi''_2(\alpha_j)+\lambda_j^2\phi''_1(\alpha_j)=0$
for $j=1,\dots,k;$

\noindent$\bullet$
$\phi'_3(\alpha_j)-\lambda_j\phi'_2(\alpha_j)+\lambda_j^2\phi'_1(\alpha_j)=0$
for $j=k+1,\dots,l.$
\end{proposition}

As a simple consequence of Proposition \ref{3} one may formulate a
corollary similar to Corollary \ref{2}. More philosophically,
(SNPP) is solvable on $\Om_3$ if and only if a ``modified''
spectral Nevanlinna-Pick problem ``with derivatives" can be solved
on $\G_3$.

\begin{proof} The necessary part follows by straightforward calculations. For the
converse,
similarly to the previous proof we may replace any $A_j$ by its
rational canonical form, i.e.~we may assume that
$$A_j:=\left(\begin{array}{ccc}\lambda_j&0&0\\0&0&1\\0&-\lambda_j\mu_j&\lambda_j+\mu_j\end{array}\right),
\quad k+1\le j\le l,$$
$$A_j:=\left(\begin{array}{ccc}0&1&0\\0&0&1\\c_j&-b_j&a_j\end{array}\right),
\quad l+1\le j\le m,$$ where $a_j=\lambda_j+\mu_j+\nu_j,$
$c_j=\lambda_j\mu_j+\mu_j\nu_j+\nu_j\lambda_j$ and
$c_j=\lambda_j\mu_j\nu_j.$ We shall look for a $\psi$
of the form
$$\psi=\left(\begin{array}{ccc}f_{11}&f_{12}&0\\0&f_{22}&f_{23}\\f_{31}&f_{32}&f_{33}\end{array}\right),$$
where $f_{33}=\varphi_1-f_{11}-f_{22},$ and $f_{pq}$ are entire
function with $f_{pq}(\alpha_j)=a_{pq}^j,$ such that all their
zeros are simple ones and belong to $\{\alpha_1,\dots,\alpha_m\}.$
We have to satisfy the following conditions:
$$f_{32}(\alpha_j)=0,\ 1\le j\le k,\quad  f_{31}(\alpha_j)=0,\ 1\le j\le l,$$
$$f_{32}f_{23}=\tilde g:=f_{11}f_{22}+f_{22}f_{33}+f_{33}f_{11}-\varphi_2,$$
$$f_{31}f_{12}f_{23}=\tilde h:=\varphi_3+f_{11}(\tilde g-f_{22}f_{11}).$$
Note that $\phi'_2(\alpha_j)=2\lambda_j\phi'_1(\alpha_j)$ means
$\ord_{\alpha_j}\tilde g\ge 2$ for $1\leq j\leq k$ and thus
$f_{32}:=\tilde g/f_{23}$ is a well-defined function with the
desired properties. On the other hand, the conditions
$\phi'_2(\alpha_j)=2\lambda_j\phi'_1(\alpha_j)$ and
$\phi'_3(\alpha_j)=\lambda_j^2\phi'_1(\alpha_j)$ imply that
$\phi'_3(\alpha_j)-\lambda_j\phi'_2(\alpha_j)+\lambda_j^2\phi'_1(\alpha_j)=0$,
so that the last condition holds for $1\leq j\leq l$. This gives
that $\ord_{\alpha_j}\tilde h\ge 2$ for those $j$'s.  For
$j=1,\dots,k$ we have, in addition,
$\phi''_3(\alpha_j)-\lambda_j\phi''_2(\alpha_j)+\lambda_j^2\phi''_1(\alpha_j)=0$,
so $\ord_{\alpha_j}\tilde h\geq 3$. Therefore, $f_{31}:=\tilde
h/(f_{12}f_{23})$ is a well-defined function with the desired
properties.
\end{proof}

Note that as above this result says that if there is a solution of
(SNPP), there there exists also a bounded one.

It is easy to find necessary conditions for lifting on $\Om_n$
(depending of the structure of the associated rational canonical
forms of the matrices). It would be interesting to know which of
these conditions are in fact also sufficient.
\smallskip

\begin{remark*} {\rm Let $k=2$ and denote by $l_G$ the so-called
Lempert function of a domain $G\subset\C^m$ (cf.~\cite{NTZ}).
Assume that (SNPP) with data $(\alpha_1,A_1),$ $(\alpha_2,A_2)$ is
solvable. Then
$$
l_{\Om_3}(A_1,A_2)\leq l_\D(\alpha_1,\alpha_2)=\Big
|\frac{\alpha_1-\alpha_2}{1-\overline{\alpha_1}\alpha_2}\Big |.
$$
Conversely, if $l_{\Om_3}(A_1,A_2)\leq l_\D(\alpha_1,\alpha_2),$
then the above problem is solvable. Moreover, there exists always an extremal
analytic disc through
$A_1, A_2$. Indeed, we may assume that $\alpha_1=0$. Then there
are holomorphic discs $\psi_j\in\O(\D,\Om_3)$ with
$\psi_j(0)=A_1$, $\psi_j(\alpha_{2,j})=A_2$ and
$\alpha_{2,j}\rightarrow l_{\Om_3}(A_1,A_2)=:\tilde\alpha_2\le |\alpha_2|$. In other
words, (SNPP) with data $(0,A_1),$ $(\alpha_{2,j},A_2)$ is solvable. Put
$\phi_j=\sigma\circ\psi_j$. Since $\G_3$ is a taut domain, we may assume
that $\phi_j$ converges locally uniformly to a
$\phi\in\O(\D,\G_3)$. Note that $\phi$ satisfies the necessary and
sufficient conditions for $0,\tilde\alpha_2,A_1,A_2$ from Proposition
\ref{3}. Hence we can lift $\phi$ to a $\tilde\psi\in\O(\D,\Om_3)$ with
$\tilde\psi(0)=A_1$ and $\tilde\psi(\tilde\alpha_2)=A_2$. It remains to set
$\psi(\zeta)=\tilde\psi(\tilde\alpha_2\zeta/\alpha_2).$}
\end{remark*}

Using again that $\G_3$ is a taut domain, Proposition \ref{3} implies the
following conditional stability of (SNPP).

\begin{corollary}\label{4} Let (SNPP) with data
$(\alpha_{1,j},A_{1,j}),\dots,(\alpha_{k,j},A_{k,j})\subset\D\times\Om_3$,
$j\in\N$, be solvable, and let
$(\alpha_{s,j},A_{s,j})_j\to(\alpha_s,A_s)\subset\D\times\Om_3,$
$1\le s\le k.$ Assume that:

\noindent$\bullet$ if $A_s$ is non-cyclic, then $A_{s,j}$ is
non-cyclic;

\noindent$\bullet$ if $A_s$ is scalar, then $A_{s,j}$ is scalar.

\noindent Then the problem with data
$(\alpha_1,A_1),\dots,(\alpha_k,A_k)$ is solvable.
\end{corollary}

\begin{remark*} {\rm Both conditions are essential even for
$k=2.$ For the first one, note that if $n\ge 3$ and $A_1$ is
non-scalar with equal eigenvalues, then the so-called Lempert
function $l_{\Om_n}(A_1,\cdot)$ is not continuous at some
non-scalar and non-cyclic diagonal matrix $A_2$ (see
\cite{Nik-Tho2}). For the second one, we claim that if
$A_1\in\Omega_3$ is a cyclic matrix with at least two different
eigenvalues and $A_{2,j}\in\Omega_3$ are non-scalar matrices
tending to a scalar matrix $A_2,$ then
$l_{\Om_3}(A_1,A_2)>\limsup_{j\to\infty}l_{\Om_3}(A_1,A_{2,j})$
(the same holds for $n=2,$ cf.~\cite{Nik-Tho2}). Applying an
automorphism of $\Omega_n$ of the form
$$\Phi_\lambda(X)=(X-\lambda I)(I-\overline{\lambda}X)^{-1},$$
we may assume that $A_2=0.$ Having in mind Proposition 3 and the
equality $l_{\Om_n}(A_1,0)=r(A_1)$ (cf.~\cite{NTZ}), it is enough
to show that if $A_1\in\Omega_n$ and
$$r(A_1)=\inf\{|\alpha|:\exists \psi\in\O(\D,\G_n):\psi(0)=0,
\psi(\alpha)=\sigma(A_1),\psi_n'(0)=0\},$$ then $A_1$ has equal
eigenvalues. If the characteristic polynomial of $A_1$ does not
have the form $x^n+a,$ then this follows as in \cite[Proposition
10 (iii)]{NTZ} with an obvious modification in the proof of
\cite[Lemma 11]{NTZ}. Otherwise, it suffices to take
$\psi(\zeta)=(0,\dots,0,\zeta^2).$}
\end{remark*}

\section{The spectral Carath\'eodory-Fej\'er problem}

In the second part of this note we will treat the
\textit{spectral Carath\'eodory-Fej\'er problem} which has been also discussed in
\cite{HMY}.
The simplest case of this problem is the following one (SCFP):
\smallskip

\noindent\textit{Given $A\in\Om_n$ and $B\in\M_n,$ determine
whether there exists a $\psi\in\O(\D,\Om_n)$ such that $\psi(0)=A$
and $\psi'(0)=B.$}
\smallskip

When $A$ is cyclic, for any
$\phi\in\O(\D,G_n)$ with $\phi(0)=\sigma(A)$ and
$\phi'(0)=\sigma'(0)B$ ($\sigma'(A)$ is the Fr\'echet derivative
of $\sigma$ at $A$) there is a $\psi\in\O(\D,\Om_n)$ such that $\phi=\sigma\circ\psi,$ $\psi(0)=A$ and $\psi'(0)=B$ (see
\cite{HMY}).
 Denote by $\kappa_G$ the Kobayashi pseudometric of a domain $G$ in
$\Bbb C^m$ (cf.~\cite{NTZ}).
 Then $\kappa_{\Om_n}(A;B)=\kappa_{\mathbb G_n}(\sigma(A); \sigma'(A)B)$.

As in the case of the Lempert function (see Remark after the proof of Proposition \ref{3}), extrema are attained and
a $\psi$ as in (SFCP) exists if and only if $\kappa_{\Om_n}(A;B)\le 1.$

 Note that $\kappa_{\Om_n}(A;B)=0$ if and only if
$\sigma'(A)B=0,$  if and only if there is a $\psi\in\O(\C,\Om_n)$
with $\psi(0)=A_{\mu}$ and $\psi'(0)=B$ (see \cite{Nik-Tho1}).

The situation is more complicated if $A$ is non-cyclic.

\subsection{The case of scalar matrices}

Let first $A=\lambda I.$ Applying $\Phi_{\lambda},$ we assume that
$A=0.$ Since $\Om_n$ is a pseudoconvex balanced domain, (SCFP) is
solvable if and only if $r(B)=\kappa_{\Om_n}(0;B)\le 1.$ On the
other hand, we have the following.

\begin{proposition}\label{8} Let $B\in\Om_n$ be cyclic and let
$\phi\in\O(\D,\G_n)$  be such that $\phi(0)=0.$ Then there
exists a $\psi\in\O(\D,\Om_n)$ satisfying $\phi=\sigma\circ\psi,$
$\psi(0)=0$ and $\psi'(0)=B$ if and only if
$\ord_0\phi_j\ge j$ and
$\phi_j^{(j)}(0)/j!=\sigma_j(B)$ for any $j=1,\dots,n.$
\end{proposition}

Since $\sigma'(0)B=(\tr B,0,\dots,0),$ there are $n(n-1)/2$ additional
conditions for lifting.

\begin{proof} The necessary part is obvious. For the converse, let
$$\tilde B:=
\left(\begin{array}{ccccc}
0&1&\dots&0\\
\\
0&0&\dots&1\\
b_n&b_{n-1}&\dots&b_1
\end{array}\right).
$$
be the companion matrix of $B$ and let $B=P^{-1}\tilde BP.$
Setting $\tilde\phi_j(\zeta)=\phi_j(\zeta)/\zeta^{j-1}$ and
$$\tilde\psi=
\left(\begin{array}{ccccc}
0&\zeta&\dots&0\\
\\
0&0&\dots&\zeta\\
\tilde\phi_n&\tilde\phi_{n-1}&\dots&\tilde\phi_1
\end{array}\right),
$$
then the mapping $\psi=P^{-1}\tilde\psi P$ does the job.
\end{proof}

To complete the picture about the lifting property if $n=2,$ it remains
to show the following.

\begin{proposition}\label{9} Let $B=\lambda I\in\M_2$ and
$\phi\in\O(\D,\G_2)$ be such that $\phi(0)=0.$ Then there
exists a $\psi\in\O(\D,\Om_2)$ satisfying $\phi=\sigma\circ\psi,$
$\psi(0)=0$ and $\psi'(0)=B$ if and only if $\phi_1'(0)=2\lambda,$
$\phi_2'(0)=0,$ $\phi_2''(0)=2\lambda^2,$ and $\phi_2'''(0)=3\lambda\phi_1''(0).$
\end{proposition}

So we need more additional conditions for the lifting. As one may expect (see also below)
these conditions depend on the structure of the rational canonical forms of both $A$ and $B.$

\begin{proof} The necessary part is clear. For the converse, set $f_{22}=\phi_1-\lambda\zeta$
and $f_{21}=\lambda f_{22}/\zeta-\phi_2/\zeta^2.$ Then the mapping
$$\tilde\psi=\left(\begin{array}{ccc}\lambda\zeta&\zeta^2\\f_{21}&f_{22}\end{array}\right)$$
is well-defined and fulfills the required properties.
\end{proof}

In the case of $n=3$ and $B=\lambda I$ we have the following result.

\begin{proposition}\label{10}
Let $B=\lambda I\in\M_3$ ($\lambda\neq 0$) and
$\phi\in\O(\D,\G_3)$ with $\phi(0)=0$. Then there exists a
$\psi\in\O(\D,\Om_3)$ with $\phi=\sigma\circ\psi$, $\psi(0)=0$ and
$\psi'(0)=B$ if and only if $\phi_1'(0)=3\lambda$, $\phi_2'(0)=0$,
$\phi_2''(0)/2=3\lambda^2$, $\phi_3'(0)=\phi_3''(0)=0$,
$\phi_3'''(0)/3!=\lambda^3$, $\phi_2'''(0)/3!=\lambda\phi_1''(0)$,
$\phi_3^{(4)}(0)/4!=\lambda^2\phi_1''(0)/2,$ and
$\phi_3^{(5)}(0)/5!-\lambda\phi_2^{(4)}/4!+\lambda^2\phi_1'''(0)/3!=0.$
\end{proposition}

We should mention that a similar result is true for arbitrary $n$ and $B=\lambda I\in\M_n$.

\begin{proof} Straightforward but tedious calculations lead to the necessary conditions
(expand $\psi_{ii}$ up to order 3 and $\psi_{ij}$ ($i\neq j$) up
to order 2). So it remains to prove the converse statement. Let
$\phi$ be given as in the proposition. We are looking for the
following mapping $\psi$ as a lifting of $\phi$:
$$
\psi=\left(\begin{array}{ccc}f_{11}&f_{12}&0\\
0&f_{22}&f_{23}\\
f_{31}&f_{32}&f_{33}
\end{array}\right),
$$
where $f_{11}(\zeta)=f_{22}(\zeta):=\lambda\zeta$,
$f_{12}(\zeta)=f_{23}(\zeta):=\zeta^2$, and
$f_{33}:=\phi_1-f_{11}-f_{22}$. Moreover, put
$$f_{32}=\frac{f_{11}f_{22}+f_{22}f_{33}+f_{33}f_{11}-\phi_2}{\zeta^2},$$
$$f_{31}=\frac{\phi_3-f_{11}f_{22}f_{33}+f_{11}f_{23}f_{32}}{\zeta^4}.$$
Using the conditions on $\phi$ it turns out that
$\psi\in\O(\D,\Om_3)$ is a well-defined mapping satisfying all
desired conditions.
\end{proof}

To finish with the lifting property if $n=3$ and $A=0,$ it remains to
consider the case, when $B$ is a non-cyclic and non-scalar matrix.

\begin{proposition}\label{11} Let $B\in\Om_3$ be a non-cyclic and non-scalar matrix
such that $\sp(B)=\{\lambda,\lambda,\mu\}.$ Let
$\phi\in\O(\D,\G_3)$  be such that $\phi(0)=\sigma(0).$ Then there
exists a $\psi\in\O(\D,\Om_3)$ satisfying $\phi=\sigma\circ\psi,$
$\psi(0)=0$ and $\psi'(0)=B$ if and only if
$\phi_1'(0)=2\lambda+\mu$, $\phi_2'(0)=0$,
$\phi_2''(0)/2=\lambda^2+2\lambda\mu$, $\phi_3'(0)=\phi_3''(0)=0$,
$\phi_3'''(0)/3!=\lambda^2\mu$, and $\phi_3^{(4)}(0)/4!-\lambda
\phi_2'''(0)/3!+\lambda^2\phi_1''(0)/2=0.$
\end{proposition}

\begin{proof} Applying an automorphism of $\Omega_3$ of the form $X\to
P^{-1}XP,$ we may assume that
$B=\left(\begin{array}{ccc}\lambda&0&0\\0&\lambda&1
\\0&0&\mu\end{array}\right).$
The necessary part is almost trivial (use that $\phi_2'''(0)/3=
(\lambda+\mu)\phi_1''(0)+(\lambda-\mu)\psi_{33}''(0)-\psi_{32}''(0)$
and
$\phi_3^{(4)}(0)/12=\lambda\big(\mu\phi_1''(0)+(\lambda-\mu)\psi_{33}''(0)-\psi_{32}''(0)\big)$).
For the converse, set $f_{11}=f_{22}=\lambda\zeta$, $f_{33}:=\phi_1-f_{11}-f_{22}$,
$f_{12}=\zeta^2,$ $f_{23}=\zeta,$
$f_{32}=2\lambda\phi_1-3\lambda^2\zeta-\phi_2/\zeta,$ and
$$f_{31}=\frac{\phi_3-\lambda\zeta(\phi_2-\lambda\zeta\phi_1+\lambda^2\zeta^2)}{\zeta^3}.$$
Then the mapping $$\psi=\left(\begin{array}{ccc}f_{11}&f_{12}&0\\0&f_{22}&f_{23}\\
f_{31}&f_{32}&f_{33}\end{array}\right)$$ is well-defined and has
the required properties.
\end{proof}

Note that  $\kappa_{\Bbb G_n}(0;\sigma'(0)B)=|\tr B|/n$, while
$\kappa_{\Omega_n}(0; B)= r(B)$, and these two quantities only match when
all eigenvalues of $B$ coincide.  When that is not the case, considering a
sequence $A_j \to 0$ of cyclic matrices, for which $\kappa_{\Bbb
G_n}(\sigma(A_j);\sigma'(A_j)B)=\kappa_{\Omega_n}(A_j; B)$,
shows that (SCFP) is not stable
in the first variable. An example in \cite{Nik-Tho1} shows that (SCFP)
is also not stable in the second variable for $n=3$ (see also below).

\subsection{The case of non-cyclic and non-scalar matrices}

To complete the picture about the lifting property if $n=3,$ it
remains to settle the case, when $A\in\Omega_3$ is a non-cyclic
and non-scalar matrix. Applying automorphisms of $\Omega_3$ of the
forms $\Phi_\lambda$ and $X\to P^{-1}XP,$ we may assume that
$A=A_\mu:=\left(\begin{array}{ccc}0&0&0\\0&0&1\\0&0&\mu\end{array}\right).$

\begin{proposition}\label{5} Let $B\in\Om_3$ with $b_{12}\neq 0,$ or $b_{31}\neq\mu
b_{21},$ or $b_{32}\neq\mu(b_{22}-b_{11}).$ Let
$\phi\in\O(\D,\G_3)$ be such that $\phi(0)=\sigma(A_\mu)$ and
$\phi'(0)=\sigma'(A_\mu)B.$ Then there exists a
$\psi\in\O(\D,\Om_3)$ satisfying $\phi=\sigma\circ\psi,$
$\psi(0)=A_\mu$ and $\psi'(0)=B$ if and only if
$$\frac{\phi''_3(0)}{2}=\mu\left|\begin{array}{cc}b_{11}&b_{12}\\b_{21}&b_{22}\end{array}\right|-
\left|\begin{array}{cc}b_{11}&b_{12}\\b_{31}&b_{32}\end{array}\right|.$$
\end{proposition}

Note that the right-hand side of the last equality is the
second-order G\^ateaux derivative $\sigma_3''(A;B).$

\begin{proof} The necessary part follows by straightforward calculations.

For the converse, replacing the wanted $\psi$ by
$\tilde\psi=e^{-\zeta X}\psi e^{\zeta X},$  we obtain
$\tilde B:=\tilde\psi'(0)=B+[A_\mu,X].$  Choosing an appropriate $X,$ we may
assume that
$$\tilde B=\left(\begin{array}{ccc}\tilde b_{11}&\tilde b_{12}&0\\0&0&0\\ \tilde b_{31}& \tilde b_{32}& \tilde b_{
33}\end{array}\right),$$ where $\tilde b_{1j}=b_{1j}$, $j=1, 2$,
$\tilde b_{3j}=b_{3j}-\mu b_{2j}$, $j=1, 2$, and $\tilde b_{33}=b_{33}+b_{22}$.   From now
on we write $b_{ij}$ for $\tilde b_{ij}$.
Then $\phi(0)=(\mu,0,0),$ $\phi'(0)=(b_{11}+b_{33},\mu
b_{11}-b_{32},0)$ and $\phi''_3(0)/2= b_{31}b_{12}-b_{11}b_{32}.$

Let, for example, $b_{12}\neq 0.$ Set $f_{11}=b_{11}\zeta,$
$f_{12}=b_{12}\zeta,$ $f_{33}=\varphi_1-f_{11}$ and
$f_{32}=e^{\zeta^2}(f_{11}f_{33}-\phi_2).$ It follows that
$f_{31}=\frac{f_{11}f_{32}+e^{\zeta^2}\phi_3}{f_{12}}$ is a
well-defined function with $f_{31}(0)=0$ and $f_{31}'(0)=b_{31}.$
Then the mapping
$$\psi=\left(\begin{array}{ccc}f_{11}&f_{12}&0\\0&0&e^{-\zeta^2}\\f_{31}
&f_{32}&f_{33}\end{array}\right)$$ does the job. The case
$b_{31}\neq 0$ is similar.

Let now $b_{12}=b_{31}=0$ but $b_{32}\neq-\mu b_{11}.$ The same
mapping with $f_{11}=b_{11}\zeta+c\zeta^2$ and $f_{12}=\zeta^2$
has the desired properties, where
$$c=\frac{1}{b_{32}+\mu b_{11}}\left(b_{11}\left(\frac{\phi_2''(0)}{2}-b_{11}b_{33}\right)
-\frac{\phi_3'''(0)}{3!}\right).$$
\end{proof}

\begin{proposition}\label{6} Let $B\in\Om_3$ with $b_{12}=0,$ $b_{31}=\mu
b_{21}$ and $b_{32}=\mu(b_{22}-b_{11}).$ Let $\phi\in\O(\D,\G_3)$
be such that $\phi(0)=\sigma(A_\mu)$ and
$\phi'(0)=\sigma'(A_\mu)B.$ Then there exists a
$\psi\in\O(\D,\Om_3)$ satisfying $\phi=\sigma\circ\psi,$
$\psi(0)=A_\mu$ and $\psi'(0)=B$ if and only if $\phi_3''(0)/2=\mu
b_{11}^2$ and $\phi_3'''(0)/3!=b_{11}(\phi_2''(0)/2-b_{11}b_{33})$
\end{proposition}

\begin{proof} The necessary part is straightforward. For the converse, we may assume
that $b_{ij}=0$ for $(i,j)\neq(1,1),(3,2),(3,3).$ Then the same
mapping as in the first part of the proof of Proposition \ref{5}
with $f_{12}=\zeta^2$ does the job.
\end{proof}

We should mention that all the above conditions may be expressed
in terms of the original matrices $A$ and $B$. Since these
calculations look awful we decided not to include them.

\begin{remark*}{\rm We see that (SCFP) with data $(A_0,B)$ is not necessarily
stable near $B.$ To confirm this, note that
$\kappa_{\Om_3}(A_0;B)=0$ if and only if $B\in\mathcal
B_1\setminus\mathcal B_2,$ where
$$\begin{aligned} \mathcal B_1:&=\{B\in\M_3:\tr
B=b_{32}=b_{12}b_{31}=0\}\\
&=\{B\in\M_3:\sigma'(A_0)B=0,\sigma''_3(A_0;B)=0\}
\end{aligned}$$
 and $\mathcal
B_2:=\{B\in\M_3:b_{11}\neq 0,b_{12}=b_{31}=0\}$ (see
\cite{Nik-Tho1}). Moreover, if $B\in\mathcal B_1\setminus\mathcal
B_2,$ then there exists a $\psi\in\O(\C,\Om_3)$ with $\psi(0)=A_0$
and $\psi'(0)=B$ (see \cite{Nik-Tho1} or replace $\phi$ by  the constant map
$\phi(\zeta)=0$ in the construction of $\psi$ in Proposition \ref{5}
and \ref{6}).
On the other hand, it follows by the proof of
\cite[Example 2]{Nik-Pfl} that if $B\in\mathcal B_1\cap\mathcal
B_2,$ then there is a $\psi\in\O(\D,\Om_3)$ with $\psi(0)=A_0$ and
$\psi'(0)=B$ if and only if $\kappa_{\Om_3}(A_0;B)=|b_{11}|\le 1.$
This phenomenon agrees with Propositions \ref{5} and \ref{6}.

We have no such an effect if $\mu\neq 0.$ Then
$\kappa_{\Om_3}(A_\mu;B)=0$ if and only if $\sigma'(A_\mu)B=0$ and
$\sigma''_3(A_\mu;B)=0,$ if and only if there is a
$\psi\in\O(\C,\Om_3)$ with $\psi(0)=A_{\mu}$ and $\psi'(0)=B$ (see
\cite{Nik-Tho1}, where $A_\mu$ is replaced by the similar matrix
$\diag(0,0,\mu),$ or replace $\varphi$ by
the constant map $\phi(\zeta)=(\mu,0,0)$ as above).
This also agree with Propositions \ref{5} and \ref{6} (since if $\phi_3''(0)=0,$ then
$\phi_3'''(0)=0$).}
\end{remark*}

\end{document}